\documentclass[12pt]{article}

\usepackage{diagbox}
\usepackage{mathtools}
\usepackage{bbm}
\usepackage{latexsym}
\usepackage{epsfig}
\usepackage{amsmath,amsthm,amssymb,enumerate}
\usepackage[margin=1in,marginparwidth=0.75in]{geometry}

\usepackage{hyperref}

\usepackage{cleveref}

\newcommand{\scr}{\mathcal}
\newcommand{\mb}{\mathbb}
\newcommand{\til}{\widetilde}

\newcommand{\eps}{\epsilon}

\renewcommand{\Pr}{\mb{P}}

\usepackage{color}

\newtheorem{theorem}{Theorem}

\newtheorem{corollary}[theorem]{Corollary}

\theoremstyle{definition}
\newtheorem{Remark}[theorem]{Remark}

\newcommand{\mc}[1]{\mathcal{#1}}
\newcommand{\rbrac}[1]{\left(#1\right)}
\newcommand{\cbrac}[1]{\left\{ #1\right\}}

\newcommand{\sbrac}[1]{\left[ #1\right]}
\newcommand{\tbf}[1]{\textbf{#1}}

\def\E{\mathbb{E}}
\def\Var{\mbox{{\bf Var}}}
\def\Pr{\mathbb{P}}

\def\nn{\nonumber}
   \def\D{\Delta}
\def\eps{\varepsilon}

\def\t{\tau}

\def\yt{{\tilde y}}

\title{
Extending Wormald's Differential Equation Method to One-sided Bounds}

\date{}

\begin{document}

\author{
Patrick Bennett\thanks{Department of Mathematics, Western Michigan University (\texttt{patrick.bennett@wmich.edu}). Supported in part by Simons Foundation Grant \#426894.}
\and
Calum MacRury
\thanks{Graduate School of Business, Columbia University, New York, USA,
\texttt{cm4379@columbia.edu}. The majority of the author's work done on this paper was while they were affiliated with the Department of Computer Science at the University of Toronto.}
}
\maketitle

\begin{abstract}
    In this note, we formulate a ``one-sided'' version of Wormald's differential equation method. In the standard ``two-sided'' method, one is given a family of random variables which evolve over time
    and which satisfy some conditions including a tight estimate of the expected change in each variable over one time step. These estimates for the expected one-step changes suggest that the variables ought to be close to the solution of a certain system of differential equations, and the standard method concludes that this is indeed the case. We give a result for the case where instead of a tight estimate for each variable's expected one-step change, we have only an upper bound. Our proof is very simple, and is flexible enough that if we instead assume tight estimates on the variables, then we recover the conclusion of the standard differential equation method.

\end{abstract}

\section{Introduction} \label{sec:intro}

In the most basic setup of Wormald's differential equation method,
one is given a sequence of random variables $(Y(i))_{i=0}^{\infty}$ derived from some random process
which evolves step by step. The random variables $(Y(i))_{i=0}^{\infty}$ all implicitly depend on some $n \in \mb{N}$, and the goal is understand their typical behaviour as $n \rightarrow \infty$. 

Our running example is based on
the Erdős–Rényi \textbf{random graph process} $(G_i)_{i=0}^{m}$ on vertex set $[n] :=\{1,\ldots,n\}$ where 
$G_{i}=([n], E_i)$ and $m,n \in \mb{N}$. Here $G_0 = ([n], E_0)$ is the empty graph, and $G_{i+1}$ is constructed
from $G_i$ by drawing an edge $e_{i+1}$ from $\binom{[n]}{2} \setminus E_i$ uniformly at random (u.a.r.), and
setting $E_{i+1} := E_{i} \cup \{e_{i+1}\}$. 
Suppose that we wish to understand the size of the matching produced by
the greedy algorithm as it executes on $(G_i)_{i=0}^m$. More specifically, when $e_{i+1}$ arrives, the greedy algorithm adds $e_{i+1}$ to the current
matching if the endpoints of $e_{i+1}$ were not previously matched. We will let $m=cn$, i.e. we will add a linear number of random edges. Observe that if $Y(i)$ 
is the number of edges of $G_i$ matched by the algorithm,
then $Y(i)$ is a function of $e_1, \ldots ,e_i$ (formally, $Y(i)$ is $\scr{H}_i$-measurable where $\scr{H}_i$ is the sigma-algebra generated from $e_1, \ldots ,e_i$). Then for $i < m$, 
\begin{equation} \label{eq:greedy_expected_changes}
    \mb{E}[ \Delta Y(i) \mid \scr{H}_i] =  \frac{\binom{n-2Y(i)}{2}}{\binom n2 - i} = \left(1 - \frac{2 Y(i)}{n}\right)^{2} + O\rbrac{\frac 1n },
\end{equation}
where $\Delta Y(i):= Y(i+1) - Y(i)$, and the asymptotics are as $n \rightarrow \infty$ (which will be the case for the remainder of this note). By scaling $(Y(i))_{i=0}^{m}$ by $n$, we can interpret the left-hand side of
\eqref{eq:greedy_expected_changes} as the ``derivative'' of $Y(i)/n$ evaluated
at $i/n$. This suggests the following \textit{differential equation}:
\begin{equation} \label{eqn:greedy_differential_equation}
    y'(t) = (1-2y(t))^2, \qquad y(0)=0
\end{equation}
with initial condition $y(0)=0$. Wormald's differential equation method allows us to conclude that \textbf{with high probability} (i.e.~with probability tending to $1$ as $n \rightarrow \infty$, henceforth abbreviated w.h.p.),
\begin{equation} \label{eq:greedy_performance}
    Y(m) = (1 + o(1)) y(m/n),
\end{equation}
where $y(t) := t/(1+2t)$ is the unique solution to \eqref{eqn:greedy_differential_equation}.

Returning to the general setup of the differential equation method, suppose we are given a sequence of random variables $(Y(i))_{i=0}^{\infty}$ which implicitly depend on $n \in \mb{N}$. 
Assume that the one-step changes are bounded, i.e., there exists a constant $\beta \ge 0$
such that $|\Delta Y(i)| \le \beta$ for each $i \ge 0$.
Moreover, suppose each $Y(i)$ is determined by some sigma-algebra $\scr{H}_i$,
and its \textit{expected} one-step changes are described by some Lipshitz
function $f=f(t,y)$. That is, for each $i \ge 0$,
\begin{equation}\label{eq:trend}
    \mb{E}[\Delta Y(i) \mid \scr{H}_i] = f(i/n, Y(i)/n) + o(1).
\end{equation}
If $Y(0) = (1 +o (1)) \til{y} n$ for some constant $\til{y}$, and $m=m(n)$ is not too large,
then the differential equation method allows us to conclude that w.h.p. $Y(m) /n = (1 + o(1)) y(m/n)$ for $y$ which satisfies the differential equation suggested by \eqref{eq:trend}, i.e. 
\[
y'(t) = f(t, y(t))
\]
with initial condition $y(0) = \til{y}$.
In this note, we consider the case when we have an inequality in place of \eqref{eq:trend}. We are motivated by applications to online algorithms in which one wishes to upper bound the performance of \textit{any} online algorithm, opposed to just a \textit{particular} algorithm. (See \Cref{sec:toy_example} for an example pertaining to online matching
in $(G_i)_{i=0}^m$ as well as some discussion of further applications). We are also motivated by the existence of deterministic results of which we wanted to prove a random analogue. For example, consider
the following classical result due to  Petrovitch \cite{MichelPetrovitchSurUM}:
\begin{theorem}\label{thm:comparison1D}
Suppose $f:\mathbb{R}^2 \rightarrow \mathbb{R}$ is Lipschitz continuous, and $y=y(t)$ satisfies 
\begin{equation*}
    y'(t) = f(t, y(t)), \qquad y(c) = y_0. 
\end{equation*}
Suppose $z=z(t)$ is differentiable and satisfies 
\begin{equation*}
    z'(t) \le f(t, z(t)), \qquad z(c) = z_0 \le y_0. 
\end{equation*}
Then $z(t) \le y(t)$ for all $t \ge c$. 
\end{theorem}
With the above result in mind (as well as the standard differential equation method), it's natural to wonder what can be said about a sequence of random variables $(Z_i)_{i=1}^\infty$ satisfying
\begin{equation*}
    \mb{E}[\Delta Z_i \mid \scr{H}_i] \le f(i/n, Z_i/n)
\end{equation*}
instead of the equality version \eqref{eq:trend}. More precisely, if $(Y_i)_{i=1}^\infty$ satisfies \eqref{eq:trend} and $Z_0 < Y_0$ then should it not follow that we likely have $Z_i \le Y_i$ (perhaps modulo some relatively small error term) for all $i \ge 0$?

We briefly point out that if $f$ is nonincreasing in its second variable, then the problem described in the previous paragraph is much easier. Indeed, whenever the random variable satisfies $Z_i - Y_i \le 0$, it is also a supermartingale. More precisely, when $Z_i \le Y_i$ we have that 
\[
\mb{E}[\Delta (Z_i - Y_i) \mid \mc{H}_i] \le f(i/n, Z_i/n) - f(i/n, Y_i/n) \le 0
\]
by the monotonicity assumption. In this case, assuming the initial gap $|Z_0-Y_0|$ is large enough, standard martingale techniques can be used to bound the probability that the supermartingale $Z_i - Y_i$ becomes positive. However, we would like to handle applications where we do not have this monotonicity assumption. For instance, in our running example, $f(t,z) = (1 -2z)^2$ is \textit{not} increasing in $z$.

Of course, the differential equation method in general deals with systems of random variables (and the associated systems of differential equations). So what can be said about systems of deterministic functions whose derivatives satisfy inequalities instead? It turns out that to generalize Theorem \ref{thm:comparison1D} to a system, we need the functions
to be \textbf{cooperative}. We say the functions $f_j:=\mathbb{R}^{a+1} \rightarrow \mathbb{R}, 1 \le j \le a$ are cooperative (respectively, \textbf{competitive}) if each $f_j$ is nondecreasing (respectively, nonincreasing) in all of its $a+1$ inputs except for possibly the first input and the $(j+1)^{th}$ one. In other words, $f_j(t, y_1, \ldots y_a)$ is nondecreasing in all variables except possibly $t$ and $y_j$. Note
that some sources refer to a system with the cooperative property as being \textbf{quasimonotonic}. Observe that in the one-dimensional case $a=1$, every function is cooperative/cooperate. The following theorem is folklore (see \cite{walter} for some relevant background, and \Cref{sec:proving_one_sided} for a proof):

\begin{theorem}\label{thm:comparisonmulti}
Suppose $f_j:\mathbb{R}^{a+1} \rightarrow \mathbb{R}, 1 \le j \le a$ are Lipschitz continuous and cooperative, and $y_j$ satisfies 
\begin{equation*} \label{eq:ydiffeq}
    y_j'(t) = f_j(t, y_1(t), \ldots, y_a(t)),\qquad 1 \le j \le a,\quad  t \ge c.
\end{equation*}
Suppose $z_j, 1 \le j \le a$ are differentiable and satisfy $z_j(c) \le y_j(c)$ and
\begin{equation*} \label{eq:zdiffeq}
    z_j'(t) \le f_j(t, z_1(t), \ldots, z_a(t)),  \qquad 1 \le j \le a,\quad t \ge c.
\end{equation*}
Then $z_j(t) \le y_j(t)$ for all $1 \le j \le a$, $t \ge c.$ 
\end{theorem}
Cooperativity is necessary in the sense that if we do not have it, then one can choose initial conditions for the functions $y_j, z_j$ to make the conclusion of \Cref{thm:comparisonmulti} fail. Indeed, suppose we do not have cooperativity, i.e. there exist $j, j'$ with $j' \neq j+1$ and some points ${\bf p}, {\bf p}' \in \mathbb{R}^{a+1}$  that agree everywhere except for their $j'^{th}$ coordinate, where we have $p_{j'} > p_{j'}'$, and $f_j({\bf p}) < f_j(\bf{p'})$. Consider the following initial conditions: 
\[
\Big(c, y_1(c), \ldots, y_a(c)\Big) = {\bf p}, \qquad \Big(c, z_1(c), \ldots, z_a(c)\Big) = {\bf p}'.
\]
Then we have that $z_j(c) = y_j(c) = p_{j+1} = p'_{j+1}$. Furthermore, $z_j'(c)$ could be as large as $f_j({\bf p}') > f_j({\bf p}) = y_j'(c)$ in which case clearly $z_j(t) > y_j(t)$ for some $t>c$. 

Our main theorem in this paper, \Cref{thm:de_one_sided}, is essentially the random analogue of \Cref{thm:comparisonmulti}.
Before providing its formal statement, we expand  upon why it
is useful for proving \textit{impossibility/hardness} results for online algorithms. The reader
can safely skip \Cref{sec:toy_example} if they would first like to instead read \Cref{thm:de_one_sided}.

\subsection{Motivating Applications} \label{sec:toy_example}


The example considered in this section is closely related to the $1/2$-impossibility (or hardness) result for an online stochastic matching problem considered by the second author, Ma and Grammel in Theorem $5$ of \cite{MacRuryMG24}. In fact, 
\Cref{thm:de_one_sided} is used explicitly to prove this result, and it is also used in a similar way in Theorem $1.4$ of \cite{macruryinduction2023}. Our theorem can also be used to simplify
the proofs of the $\frac{1}{2}(1 + e^{-2})$-impossibility result of Fu, Tang, Wu, Wu and Zhang (Theorem $2$ in \cite{Fu2021}),
and the $1-\ln(2-1/e)$-impossibility result of Fata, Ma and Simchi-Levi (Lemma $5$ in \cite{Fata2019}).
All of the aforementioned papers prove impossibility results for various online stochastic optimization problems -- more specifically, hardness results for \textit{online contention resolution schemes} 
\cite{feldman2021online} or \textit{prophet inequalities} against an ``ex-ante relaxation'' \cite{Lee2018}. We think that \Cref{thm:de_one_sided} will find further applications as a technical
tool in this area.

Let us now return to the definition of the Erdős–Rényi random graph process $(G_i)_{i=0}^m$ as discussed in \Cref{sec:intro},
where we again assume that $m = c n$ for some constant $c > 0$.
Recall that \eqref{eq:greedy_performance} says that if $Y(m)$ is the size of the matching constructed by the
greedy matching algorithm when executed on $(G_i)_{i=0}^m$, then w.h.p. $Y(m)/n = (1 + o(1)) y(c)$ where $y(c) =  c/(1+2c)$. In fact, \eqref{eq:greedy_performance} can be made
to hold with probability $1 - o(1/n^2)$, and so $\mb{E}[ Y(m) ]/n = (1 +o(1)) c/(1+2c)$
after taking expectations.

The greedy matching algorithm
is an example of an \textbf{online (matching) algorithm} on $(G_i)_{i=0}^m$. An online algorithm begins with the empty matching on $G_0$, and its goal is to
build a matching of $G_{m}$. While it knows the distribution of $(G_i)_{i=0}^m$ upfront, it learns the \textit{instantiations} of the edges sequentially and must execute online. Formally, in each step $i \ge 1$, it
is presented $e_i$ and then makes an irrevocable decision as to whether or not to include $e_i$ in its current matching, based upon $e_1, \ldots , e_{i-1}$ and its previous matching decisions. Its output is the matching $M_m$, and its goal is to maximize $\mb{E}[ |M_m|]$. Here the expectation is over $(G_i)_{i=1}^{m}$ and any randomized decisions made by the algorithm.

Suppose that we wish to prove that the greedy algorithm is asymptotically optimal.
That is, for \textit{any} online algorithm, if $M_m$ is the matching it outputs
on $G_m$, then $\mb{E}[ |M_m|] \le (1 +o(1)) \mb{E}[ Y(m)]$. In order to prove this
directly, one must compare the performance of any online algorithm to the greedy algorithm.
This is inconvenient to argue, as there exist rare instantiations of $(G_i)_{i=0}^{m}$
in which being greedy is clearly sub-optimal.

We instead upper bound the performance of any online algorithm by $(1 + o(1)) y(c) n$. Let $(M_i)_{i=0}^{m}$
be the sequence of matchings constructed by an \textit{arbitrary} online algorithm while executing
on $(G_i)_{i=0}^{m}$. For simplicity, assume that the algorithm is deterministic so that
$M_i$ is $\scr{H}_i$-measurable. In this case, we can replace \eqref{eq:greedy_expected_changes} with inequality. I.e.,
if $Z(i):=|M_i|$, then for $i <m$,
\begin{equation} \label{eq:algorithm_expected_changes}
    \mb{E}[ \Delta Z(i) \mid \scr{H}_i] \le \left(1 - \frac{2Z(i)}{n}\right)^{2} + O\rbrac{\frac 1n}.
\end{equation}
Recall now the intuition behind the differential equation method.
If we scale $(Z(i))_{i=0}^{n}$ by $n$, then we can interpret the left-hand side of
\eqref{eq:algorithm_expected_changes} as the ``derivative'' of $Z(i)/n$ evaluated
at $i/n$. This suggests the following differential inequality:
\begin{equation*} \label{eq:differential_inequality}
    z' \le (1 - 2z)^2,
\end{equation*}
with inital condition $z(0)=0$.
By applying \Cref{thm:de_one_sided} to  $(Z(i))_{i=0}^{m}$,
we get that
\begin{equation*} \label{eq:algorithm_dominant}
    Z(m)/n \le (1+ o(1)) y(c)
\end{equation*}
with probability $1 - o(n^{-2})$. As a result, $\mb{E}[ Z(m)] \le (1 +o(1)) y(c) n$, 
and so we can conclude that greedy is asymptotically optimal.


\section{Main Theorem}

For any sequence $(Z(i))_{i=0}^{\infty}$ of random variables and $i \ge 0$, we will use the notation
$\Delta Z(i) := Z(i+1)-Z(i)$. Note that given a \textbf{filtration} $(\scr{H}_i)_{i=0}^{\infty}$
(i.e., a sequence of increasing $\sigma$-algebras), we say that $(Z_j(i))_{i= 0}^\infty$ is \textbf{adapted}
to $(\scr{H}_i)_{i=0}^{\infty}$, provided $Z_i$ is $\scr{H}_i$-measurable for each $i \ge 0$.
Finally, we say that a stopping time $I$ is \textbf{adapted} to $(\scr{H}_i)_{i=0}^{\infty}$, provided the event $\{I=i\}$
is $\scr{H}_i$-measurable for each $i \ge 0$.

Given $a \in \mb{N}$, suppose that $\scr{D} \subseteq \mb{R}^{a+1}$ is a bounded domain,
and for $1 \le j \le a$, let $f_j: \scr{D} \rightarrow \mb{R}$. We assume that the following hold for each $j$:
\begin{enumerate}[a)]
    \item \label{item:fjLipshitz} $f_j$ is $L$-Lipschitz on $\mc{D}$, i.e. 
    \[
    |f_j(t, y_1, \ldots, y_a) - f_j(t', y'_1, \ldots, y'_a) | \le L \cdot \max\{|t-t'|, |y_1-y'_1|, \ldots, |y_a-y'_a|\}
    \]
    \item \label{item:fjbounded} $|f_j|\le B$ on $\scr{D}$, and
    \item \label{item:cooperative} the $(f_j)_{j=1}^a$ are cooperative.
\end{enumerate}
 Given $(0,\yt_1, \ldots, \yt_a) \in \scr{D}$, assume that $y_1(t), \ldots, y_a(t)$ is the (unique) solution to the system:
\begin{equation}\label{eq:ysystem}
    y_j'(t) = f_j(t, y_1(t), \ldots, y_a(t)), \qquad y_j(0)=\yt_j
\end{equation}
for all $t \in [0, \sigma]$, where $\sigma$ is any positive value. 

\begin{theorem} \label{thm:de_one_sided}
Suppose that for each $1 \le j \le a$ we have a sequence of random variables $(Z_j(i))_{i= 0}^\infty$ which is adapted to some filtration  $(\mc{H}_i)_{i=0}^\infty$. Let $n \in \mb{N}$, and $\beta, b, \lambda, \delta > 0$
be any parameters such that $\lambda \ge \max \cbrac{ \beta + B , \frac{L+BL + \delta n}{3L}}$. 
Given an arbitrary stopping time $I \ge 0$ adapted to $(\mc{H}_i)_{i=0}^\infty$, suppose that the following properties hold for each $ 0 \le i  <  \min\{I, \sigma n \}$:
\begin{enumerate}
    \item  The `Boundedness Hypothesis': $\max_{j} |\Delta Z_j(i)| \le \beta$, and $\max_{j} \mb{E}[ ( \Delta Z_{j}(i) )^2 \mid \scr{H}_i] \le b$ \label{item:boundedness}
    \item  The `Trend Hypothesis': 
    $(i/n, Z_1(i)/n, \ldots Z_a(i)/n) \in \scr{D}$ and
    $$\E[ \D Z_j(i) \mid \mc{H}_i] \le  f_j(i/n, Z_1(i)/n, \ldots Z_a(i)/n)  + \delta $$
    
    for each $1 \le j \le a$. \label{item:expected_change}
    \item The `Initial Condition': $Z_j(0) \le y_j(0) n + \lambda $ for all $1 \le j \le a$. \label{item:initial}

\end{enumerate}
Then, with probability at least $1 - 2 a \exp\left( -\frac{\lambda^2}{2( b \sigma n + 2\beta \lambda )} \right)$,
 \begin{equation*}\label{eq:mainprob}
 Z_j(i) \le ny_{j}(i/n) + 3 \lambda e^{2L i/n} 
    \end{equation*}
for all $1 \le j \le a$ and $0 \le i \le \min\{I, \sigma n\}$.

\end{theorem}
\begin{Remark}[Simplified Parameters]
By taking $b = \beta^2$ and $I = \lceil \sigma n \rceil$, we can recover a simpler version of the theorem which is sufficient for many applications, including the motivating example of \Cref{sec:toy_example}.
\end{Remark}

\begin{Remark}[Stopping Time Selection]
Let $0 \le \gamma \le 1$ be an additional parameter to \Cref{thm:de_one_sided}.
The stopping time $I$ is most commonly applied in the following way. Suppose
that $(\scr{E}_i)_{i=0}^{\infty}$ is a sequence of events adapted
to $(\scr{H}_i)_{i=0}^{\infty}$, and for each $0 \le i < \sigma n$,
Conditions \ref{item:boundedness}.~and \ref{item:expected_change}.~are \textit{only} verified when $\scr{E}_i$ holds. Moreover, assume that $\mb{P}[\cap_{i=0}^{\sigma n -1} \scr{E}_i] = 1 -\gamma$. By defining $I$ to be the smallest $i \ge 0$ such that $\scr{E}_i$ does \textit{not} occur, \Cref{thm:de_one_sided} implies
that with probability at least $1 - 2 a \exp\left( -\frac{\lambda^2}{2( b \sigma n + 2\beta \lambda )} \right) - \gamma$,
 \begin{equation*}\label{eq:alt_prob}
 Z_j(i) \le ny_{j}(i/n) + 3 \lambda e^{2L i/n} 
\end{equation*}
for all $1 \le j \le a$ and $0 \le i \le \sigma n$.

\end{Remark}

\begin{Remark}[Competitive Functions]
Theorem \ref{thm:de_one_sided} yields upper bounds for families of random variables. There is a symmetric theorem for lower bounds, where all the appropriate inequalities are reversed and the functions $f_j$ are competitive instead of cooperative. 

\end{Remark}
We conclude the section with a corollary of \Cref{thm:de_one_sided} which provides
a useful extension of the theorem. Roughly speaking, the extension says that when
verifying Conditions \ref{item:boundedness}. and \ref{item:expected_change}. at time $0 \le i \le \min\{\sigma n, I\}$, it does
not hurt to assume that \eqref{eq:mainprob} holds.
\begin{corollary}[of \Cref{thm:de_one_sided}] \label{cor:assuming_conditions}
Suppose that in the terminology of \Cref{thm:de_one_sided}, Conditions \ref{item:boundedness}. and \ref{item:expected_change}.
are only verified for each $0 \le i \le \min\{I, \sigma n\}$ which satisfies $Z_{j'}(i) \le ny_{j'}(i/n) + 3 \lambda e^{2Li/n}$ for all $1 \le j'\le a$. In this case, the conclusion of \Cref{thm:de_one_sided} still holds.
I.e., with probability at least $1 - 2 a \exp\left( -\frac{\lambda^2}{2( b \sigma n + 2\beta \lambda )} \right)$,
 \begin{equation*}
 Z_j(i) \le ny_{j}(i/n) + 3 \lambda e^{2L i/n} 
    \end{equation*}
for all $1 \le j \le a$ and $0 \le i \le \min\{I, \sigma n\}$.

\end{corollary}

\begin{proof}[Proof of \Cref{cor:assuming_conditions}]
Let $I^*$ be the first $i \ge 0$ such that
     \begin{equation*}
 Z_{j'}(i) > ny_{j'}(i/n) + 3 \lambda e^{2L i/n} 
    \end{equation*}
for some $1 \le j' \le a$. Clearly, $I^*$ is a stopping time adapted
to $(\scr{H}_i)_{i=0}^{\infty}$. Moreover, by the assumptions of the corollary,
Conditions \ref{item:boundedness}. and \ref{item:expected_change}. hold
for each $0 \le i\le \min\{I^*, I, \sigma n\}$ and $1 \le j \le a$.
Thus, \Cref{thm:de_one_sided} implies
that with probability at least $1 - 2 a \exp\left( -\frac{\lambda^2}{2( b \sigma n + 2\beta \lambda )} \right)$,
 \begin{equation*}
 Z_j(i) \le ny_{j}(i/n) + 3 \lambda e^{2L i/n} 
    \end{equation*}
for all $1 \le j \le a$ and $0 \le i \le \min\{I,I^*, \sigma n\}$. Since the preceding event
holds if and only if $I^* > \min\{I, \sigma n\}$, the corollary is proven.
    
\end{proof}

\section{Proving \Cref{thm:de_one_sided}} \label{sec:proving_one_sided}

Before proceeding to the proof of \Cref{thm:de_one_sided}, we provide some intuition for our approach
by presenting a proof of the deterministic setting (i.e., \Cref{thm:comparisonmulti}). The notation
and structure of the proof is intentionally chosen so as to relate to the analogous approach taken in the proof of \Cref{thm:de_one_sided}.
Moreover, the main claim we prove can be viewed as an \textit{approximate} version of \Cref{thm:comparisonmulti}, in which the upper bounds on $z_{j}(0)$ and $z'_{j}$ only hold up to an additive constant $\delta > 0$.

\begin{proof}[Proof of \Cref{thm:comparisonmulti}]
Let us assume that $c=0$ is the boundary of the domain,
and $L$ is a Lipschitz constant for the cooperative functions $(f_j)_{j=1}^{a}$.
We shall prove the following:
Given an arbitrary $\delta >0$, if
\begin{equation} \label{eqn:approximate_upper_bounds_assumption}
    z_j'(t) \le f(t, z_j(t))+\delta, \qquad  z_j(0) \le y_j(0) +\delta
\end{equation}
for all $1 \le j \le a$ and $t \ge0$, then
\begin{equation} \label{eqn:approximate_upper_bounds}
    z_{j}(t) \le y_{j}(t) + \delta e^{L t}
\end{equation}
for each $1 \le j \le a$ and $t \ge 0$. 
Since \eqref{eqn:approximate_upper_bounds_assumption} holds \textit{each} $\delta > 0$ under the assumptions
of \Cref{thm:comparisonmulti}, so must \eqref{eqn:approximate_upper_bounds}. This will imply that $z_{j}(t) \le y_{j}(t)$ for each $1 \le j \le a$ and $t \ge 0$, thus completing the proof.

In order to prove that \eqref{eqn:approximate_upper_bounds_assumption} implies \eqref{eqn:approximate_upper_bounds},
define
\[
g(t):= 2\delta e^{Lt}, \qquad s_j(t):=z_j(t) - (y_j(t) +g(t)), \qquad I_j(t):=[y_j(t), y_j(t) + g(t)).
\]
It suffices to show that $\max_{1 \le j \le a} s_j(t) \le 0$ for all $t \ge 0$.
Observe first that $s_j(0)=z_{j}(0)-y_{j}(0)-g(0) \le -\delta < 0$ for all $1 \le j \le a$. Suppose for the sake of contradiction that
there exists some $1 \le j' \le a$ such that $s_{j'}(t) > 0$ for some $t > 0$. In this case, there must be some value $t_1$ with $s_{j'}(t_1)=0$ and $\max_{1 \le j \le a} s_{j}(t)<0$ for all $t < t_1$. Furthermore, there must be some $t_0<t_1$ such that $s_{j'}(t) \in [-g(t), 0)$ for all $t_0 \le t < t_1$. 
Thus, for $t_0 \le t < t_1$ we have that
\begin{equation}\label{eq:zjprimeest}
    -g(t) \le z_{j'}(t) - [y_{j'}(t) +g(t)] < 0  \qquad \Longrightarrow \qquad y_{j'}(t) \le z_{j'}(t) < y_{j'}(t) + g(t),
\end{equation}
and so
\begin{align}
    f_{j'}\Big(t, z_1(t), \ldots z_a(t)\Big) & \le f_{j'}\Big(t, y_1(t)+g(t), \ldots, z_{j'}(t), \ldots,  y_a(t)+g(t)\Big)\nn\\
    & \le f_{j'}\Big(t, y_1(t), \ldots,  y_a(t)\Big) + Lg(t)\label{eq:fjprimebound}
\end{align}
where the first line is by cooperativity of the functions $f_j$ and the second line is by the Lipschitzness of $f_{j'}$ applied to \eqref{eq:zjprimeest}. As such, for all $t \in [t_0, t_1)$,
\begin{align*}
    s_{j'}'(t)  &= z_{j'}'(t) - y_{j'}'(t) - g'(t)\\
    & = f_{j'}\Big(t, z_1(t), \ldots, z_a(t)\Big) - f_{j'}\Big(t, y_1(t), \ldots, y_a(t)\Big) -g'(t)\\
    &\le Lg(t) - g'(t) =0
\end{align*}
where the last line uses \eqref{eq:fjprimebound}. But now we have a contradiction: $s_{j'}(t_0) \in [-g(t_0), 0)$ so it is negative, $s_{j'}'(t) \le 0$ on $[t_0, t_1)$, and yet $s_{j'}(t_1)=0$. 

\end{proof}

Our proof of Theorem \ref{thm:de_one_sided} is based partly on the {\em critical interval method}. Similar ideas were used by used by Telcs, Wormald and Zhou~\cite{TAW2007} as well as Bohman, Frieze and Lubetzky~\cite{BFL2015} (whose terminology we use here). For a gentle discussion of the method see the paper of the first author and Dudek \cite{bennett_de}. Roughly speaking, the critical interval method allows us to assume we have good estimates of key variables during the very steps that we are most concerned with those variables. Historically this method has been used with more standard applications of the differential equation method in order to exploit {\em self-correcting} random variables, i.e. a variable with the property that when it strays significantly from its trajectory, its expected one-step change makes it likely to move back toward its trajectory. For such a random variable, knowing that it lies in an interval strictly above (or below) the trajectory gives us a more favorable estimate for its expected one-step change. In our setting we use the method for a similar but different reason. In particular since we can only hope for one-sided bounds, we may as well ignore our random variables when they are far away from their bounds (in any case, we do not have or need good estimates for their expected one-step changes etc. during the steps when all variables are far from their bounds).

We give an analogy. A rough proof sketch for Theorem \ref{thm:comparisonmulti} is as follows: in order to have $z_j(t) > y_j(t)$ for some $t$ there must be some time interval during which $z_j\approx y_j$ and during that interval $z_j$ must increase significantly faster than $y_j$, which contradicts what we know about their derivatives. An analogous proof sketch for Theorem \ref{thm:de_one_sided} is as follows: in order for $Z_j(i)$ to violate its upper bound, it must first enter a critical interval which we will define to be near the upper bound, and then $Z_j$ must increase significantly (more than we expect it to) over the subsequent steps, which while possible, is unlikely. Our probability bound will follow from Freedman's inequality, which we will now state.

\begin{theorem}[Freedman's Inequality \cite{freedman1975}] \label{thm:freedman}
Suppose that $(\scr{H}_i)_{i=0}^\infty$ is an increasing sequence of $\sigma$-algebras (i.e., $\scr{H}_{i-1}  \subseteq \scr{H}_{i}$ for all $i \ge 1$.)
Moreover, let $(M_i)_{i=0}^\infty$ be a sequence of random variables adapted to $(\scr{H}_i)_{i=0}^\infty$ (i.e., each $M_i$ is $\scr{H}_i$-measurable). Recall that $\Delta M_{i}:= M_{i+1} -M_i$ and $\Var(\Delta M_i \mid \scr{H}_i):= \mb{E}[ \Delta M_i^2 \mid \scr{H}_i] - \mb{E}[\Delta M_i \mid \scr{H}_i]^2$
for $i \ge 0$. Fix $m \in \mb{N}$ and $\beta,b  \ge 0$. Assume that for
each $0 \le i <m$, $\mb{E}[\Delta M_i \mid \scr{H}_i] =0$, $|\Delta M_{i}| \le \beta$, and $\Var( \Delta M_i \mid \scr{H}_i) \le b$.
Then, for any $0 < \eps < 1$,
$$
		\Pr(\exists \; 0 \le i \le m : |M_i - M_0| \geq \eps) \leq \displaystyle 2\exp\left(-\frac{\eps^2}{2(b m + \beta \eps) }\right).
$$
Moreover, if $I$ is an arbitrary stopping time adapted to $(\scr{H}_i)_{i=0}^{\infty}$ (i.e., $\{I =i\}$ is $\scr{H}_i$-measurable for each $i \ge 0$),
and the above conditions regarding $(M_i)_{i=0}^{\infty}$ are only verified for all $0 \le i < \min\{I, m\}$,
then
$$
\Pr(\exists \; 0 \le i \le \min\{I, m\} : |M_i - M_0| \geq \eps) \leq \displaystyle 2\exp\left(-\frac{\eps^2}{2(b m + \beta \eps) }\right).
$$
\end{theorem}

\begin{proof}[Proof of \Cref{thm:de_one_sided}]
Fix $0 \le i \le \sigma n$, and set $m:= \sigma n$, $t=t_i=i/n$, and $g(t):= 3\lambda e^{2Lt}$
for convenience. Define 
\begin{align*}
S_j(i):&=Z_j(i)-(ny_j(t) + g(t)), \quad X_j(i):=\sum_{k=0}^{i-1} \E[\D S_j(k) \mid \mc{H}_{k}], \\ M_j(i):&= S_j(0) + \sum_{k=0}^{i-1} \rbrac{\D S_j(k)- \E[\D S_j(k) \mid \mc{H}_{k}]},
\end{align*}
so that $S_j(i) = X_j(i) + M_j(i)$, $(M_j(i))_{i= 0}^{m}$ is a martingale and $X_j(i)$ is $\mc{H}_{i-1}$-measurable (i.e. $(X_j(i) + M_j(i))_{i= 0}^{m}$ is the Doob decomposition of $(S_j(i))_{i= 0}^{m}$). Note that we can view $S_{j}(i)/n$ as the random analogue of $s_{j}(t) =m_{j}(t) + x_{j}(t)$ from the proof of \Cref{thm:comparisonmulti}. In the previous deterministic setting, the decomposition $s_{j}(t) =m_{j}(t) + x_{j}(t)$ is redundant, as $m_{j}(t) = s_{j}(0)$, and so $x_{j}(t)$ and $s_{j}(t)$ differ by a constant. In contrast, $M_{j}(i)$ is $S_{j}(0)$, \textit{plus}
$\sum_{k=0}^{i-1} \rbrac{\D S_j(k)- \E[\D S_j(k) \mid \mc{H}_{k}]}$, the latter of which we can view
as introducing some \textit{random noise}.
We handle this random noise by showing that $M_{j}(i)$ is typically concentrated
about $S_{j}(0)$ due to Freedman's inequality (see \Cref{thm:freedman}). We refer to this as the \textit{probabilistic} part of the proof. Assuming that this concentration holds, we can upper bound $Z_{j}(i)$ by $ny_j(t) + g(t)$ via an argument which proceeds analogously to the proof of \Cref{thm:comparisonmulti}. This is
the \textit{deterministic} part of the proof.

Beginning with the probabilistic part of the proof,
we restrict our attention to $0 \le i < \min\{I, m\}$.
Observe first that
\begin{align}
    \D M_j(i) &= \D S_j(i)- \E[\D S_j(i) \mid \mc{H}_{i}]\nn\\
    & = \D [Z_j(i)-(ny_j(t) + g(t))]- \E[\D [Z_j(i)-(ny_j(t) + g(t))] \mid \mc{H}_{i}]\nn\\
    & = \D Z_j(i)- \E[\D Z_j(i) \mid \mc{H}_{i}]\nn,
\end{align}
and so by Condition \ref{item:boundedness}.~,
\[
|\D M_j(i)| \le |\D Z_j(i)| + |\E[\D Z_j(i) \mid \mc{H}_{i}]| \le 2 \beta.
\]
Also, $\Var[ \D M_{j}(i) \mid \scr{H}_i] = \mb{E}[ (\D Z_j(i)- \E[\D Z_j(i) \mid \mc{H}_{i}])^2 \mid \scr{H}_i]$. 
Thus,
\begin{align*}
    \Var[ \D M_{j}(i) \mid \scr{H}_i] &= \mb{E}[ \D Z_j(i)^2 \mid \scr{H}_i] - \E[\D Z_j(i) \mid \mc{H}_{i}]^2 \\ 
                                       &\le \mb{E}[ \D Z_j(i)^2 \mid \scr{H}_i]  \\ &\le b \tag*{by Condition \ref{item:boundedness}.~}
\end{align*}
We can therefore apply \Cref{thm:freedman} to get that
 \begin{equation}\label{eq:azuma}
\Pr(\exists \; 0 \le j \le a, 0 \le i \le \min\{I, m\} : |M_j(i) - M_j(0)| \ge \lambda) \le 2a \exp\left( -\frac{\lambda^2}{2( b m + 2\beta \lambda)} \right).
\end{equation}
Suppose the above event does \textit{not} happen i.e., for all $0 \le j \le a,0 \le i \le \min\{m,I\}$ we have that $|M_j(i) - M_j(0)| < \lambda$. We will show that we also have $Z_j(i) \le ny_j(t)+g(t)$ for all $0 \le i \le \min\{m,I\}$ and $1 \le j \le a$ (equivalently, $\max_{j} S_{j}(i) \le 0$ for all $0 \le i \le \min\{m,I\}$). This implication is the deterministic part of the proof.
By combining it with the probability bound of \eqref{eq:azuma}, this will complete the proof of \Cref{thm:de_one_sided}.

Suppose for the sake of contradiction that $i'$ is the minimal integer
such that $0 \le i' \le \min\{m,I\}$ and $Z_j(i') > ny_j(t_{i'})+g(t_{i'})$ for some $j$. Define the \tbf{critical interval}
\[
I_j(i) := \sbrac{ny_j(t), ny_j(t)+g(t)}.
\]
First observe that since $g(0):=3 \lambda > \lambda$, Condition \ref{item:initial}. implies
that $i' >0$ (and so $i'-1 \ge 0$.) We claim that $Z_j(i'-1) \in I_j(i'-1)$. Indeed, note that by the minimality of $i'$ we have that $Z_j(i'-1) \le ny_j(t_{i'-1})+g(t_{i'-1})$. On the other hand, 
$|y_j'| = |f_j| \le B$ and so each $y_j$ is $B$-Lipschitz. Thus,
since $\lambda \ge \beta + B$ (by assumption),
\begin{align}
    Z_j(i'-1) \ge Z_j(i') - \beta &> ny_j(t_{i'})+g(t_{i'}) - \beta\nn\\
    & \ge ny_j(t_{i'-1}) +3\lambda - \beta - B \nn\\
    & \ge ny_j(t_{i'-1})\nn.
\end{align}
As a result, $Z_{j}(i'-1) \in I_{j}(i'-1)$. Now let $i'' \le i'-1$ be the minimal integer with the property that for all $i'' \le i \le i'-1$, we have that $Z_j(i) \in I_j(i)$. Then $Z_j(i''-1) \notin I_j(i''-1)$ and by the minimality of $i'$ we must have that $Z_j(i''-1) < ny_j(t_{i''-1})$. By once again using the fact that $y_j$ is $B$-Lipschitz,
\begin{align}
    Z_j(i'') \le Z_j(i''-1) + \beta & < ny_j(t_{i''-1}) + \beta \le ny_j(t_{i''}) + \beta+B\label{eq:yi''bound}.
\end{align}
Now, since $Z_j(i') > ny_j(t_{i'})+g(t_{i'})$, we can apply \eqref{eq:yi''bound} to get that 
\begin{align}\label{eq:sdiff}
    S_j(i') - S_j(i'')  &= (Z_j(i') - ny_j(t_{i'}) -g(t_{i'})) - (Z_j(i'') - ny_j(t_{i''}) - g(t_{i''}))\nn \\
    &>g(t_{i''})-\beta -B \nn\\
    &\ge 3 \lambda -\beta-B.
\end{align}
Intuitively, \eqref{eq:sdiff} says that $S_{j}(i)$ increases significantly between steps $i''$
and $i'$.
We will now argue that $\mb{E}[ \Delta S_{j}(i) \mid \mc{H}_i]$ 
is nonpositive between steps $i''$ and $i'$. Informally, by scaling by $n$ and interpreting $\mb{E}[ \Delta S_{j}(i) \mid \mc{H}_i]$ 
as the ``derivative'' of $S_{j}(i)/n$ evaluated at $i/n$, this will allow us to derive a contradiction
in an analogous way as in the final part of the proof of \Cref{thm:comparisonmulti}.

Observe first that for each $i'' \le i \le i'-1$, we have that
\begin{align}
    \E[\Delta S_{j}(i) \mid \mc{H}_i] & = \E[ \D Z_{j}(i) \mid \mc{H}_i] - \D ny_j(t) - \D g(t)\nn\\
    & \le f_j(t, Z_1(i)/n, \ldots Z_a(i)/n) + \delta\nn \\
    & \qquad \qquad - f_j(t, y_1(t), \ldots, y_a(t)) + (L+BL)n^{-1} -  n^{-1} g'(t) \label{eq:mtg}
\end{align}
where the first line is by definition and line \eqref{eq:mtg} will now be justified.
The first two terms follow since by Condition \ref{item:expected_change}, $(t, Z_1(i)/n, \ldots Z_a(i)/n) \in \scr{D}$,
and
$$\E[ \D Z_j(i) \mid \mc{H}_i] \le f_j(t, Z_1(i)/n, \ldots Z_a(i)/n) + \delta.$$
For the third and fourth terms of \eqref{eq:mtg}, note that 
\begin{align}
    \D ny_j(t) = n[y_j(t_{i+1})-y_j(t_i)] &= n \int_{t_i}^{t_{i+1}} y_j'(\t) \; d\t \nn\\
    & = n \int_{t_i}^{t_{i+1}} f_j(\t, y_1(\t), \ldots, y_a(\t))  \; d\t\nn\\
    & \ge n \int_{t_i}^{t_{i+1}} f_j(t, y_1(t_i), \ldots, y_a(t_i)) -(L+BL)n^{-1}  \; d\t\label{eqn:lipschitz1}\\
    & = f_j(t, y_1(t_i), \ldots, y_a(t_i))- (L+BL)n^{-1},\nn
\end{align}
where \eqref{eqn:lipschitz1} follows since for all $\t \in [t_i, t_{i+1}]$ we have $|y_j(\t) - y_j(t_i)| \le B|\t-t_i| \le Bn^{-1}$ (since the $y_j$ are $B$-Lipschitz), and now since the $f_j$ are $L$-Lipschitz, we have
\[
|f_j(\t, y_1(\t), \ldots, y_a(\t)) - f_j(t, y_1(t_i), \ldots, y_a(t_i))| \le L \cdot \max\{|\t-t_i|, Bn^{-1}\} \le L(1+B)n^{-1}.
\]
For the last term of \eqref{eq:mtg}, we have that
\begin{align*}
    \D g(t) = 3 \lambda \rbrac{e^{2L t_{i+1}} - e^{2L t_{i}}} &= 3 \lambda e^{2L t_{i}} \rbrac{e^{2L /n} - 1} \\
    & \ge 3\lambda e^{2L t_{i}} \rbrac{\frac{2L}{n}} = n^{-1}g'(t).
\end{align*}
Observe now that by cooperativity, $f_j(t, Z_1(i)/n, \ldots Z_a(i)/n)$ is upper bounded by
\begin{equation}\label{eq:monotonicity}
 f_j\rbrac{t, \frac{ny_1(t)+g(t)}{n}, \ldots, \frac{ny_{j-1}(t)+g(t)}{n}, \frac{Z_j(i)}{n}, \frac{ny_{j+1}(t)+g(t)}{n}, \ldots, \frac{ny_a(t)+g(t)}{n}}.
\end{equation} 
Now, since $Z_j(i) \in I_j(i)$, we can apply the Lipschitzness of $f_j$
to \eqref{eq:monotonicity} to get that
\begin{equation}\label{eqn:lipschitz2}
  f_j(t, Z_1(i)/n, \ldots Z_a(i)/n) \le f_j(t, y_1(t), \ldots, y_a(t)) + L g(t)/n.  
\end{equation}
As such, applied to \eqref{eq:mtg}, 
\begin{align}
 \E[\Delta S_j(i) \mid \mc{H}_i] & \le f_j(t, Z_1(i)/n, \ldots Z_a(i)/n) + \delta - f_j(t, y_1(t), \ldots, y_a(t)) + (L+BL)n^{-1} -  n^{-1} g'(t)\nn\\
 &  \le L n^{-1}g(t) + \delta  - n^{-1} g'(t) +(L+BL)n^{-1}\nn\\
    & = L n^{-1}g(t) + \delta  - n^{-1} 2L g(t) +(L+BL)n^{-1}\nn\\
      & \le  -[L g(t)  - (L+BL + \delta n)]n^{-1}\nn\\
    & \le  -[3L \lambda  - (L+BL + \delta n)]n^{-1}\label{eq:14}\\
    & \le 0\label{eq:15},
\end{align}
where the final line follows since $\lambda \ge \frac{L+BL + \delta n}{3L}$.

Therefore, for $i'' \le i \le i'-1$ we have that
\[
0 \ge \E[\Delta S_j(i) \mid \mc{H}_i] = \E[\Delta X_j(i) \mid \mc{H}_i] + \E[\Delta M_j(i) \mid \mc{H}_i] = \Delta X_j(i)
\]
since $(M_j(i))_{i= 0}^{m}$ is a martingale and $\D X_j(i)$ is $\mc{H}_{i}$-measurable. In particular, 
\begin{equation}\label{eq:Xdiff1}
    X_j(i') \le X_j(i'').
\end{equation}
At this point, we use the event we assumed regarding $(M_{j}(i))_{i=0}^{m}$
 (directly following \eqref{eq:azuma}) to get that
\begin{equation}\label{eq:Mdiff1}
M_j(i') - M_j(i'') \le |M_j(i') - M_j(0)| + |M_j(i'') - M_j(0)| \le 2\lambda.
\end{equation}
But now we can derive our final contradiction (explanation follows):
\begin{align*}
3 \lambda -\beta-B & < S_j(i') - S_j(i'') \\
&= X_j(i') - X_j(i'') + M_j(i') - M_j(i'') \\
& \le 2 \lambda.  
\end{align*}
Indeed the first line is from \eqref{eq:sdiff}, the second is by the Doob decomposition, and the last line follows from \eqref{eq:Xdiff1} and \eqref{eq:Mdiff1}. But the last line is a contradiction since we chose $\lambda \ge  \beta+B $.

 \end{proof}

\begin{Remark}

We can relax several of the Conditions \ref{item:fjLipshitz})-\ref{item:cooperative}) on the $f_j$ and Conditions \ref{item:boundedness}. and \ref{item:expected_change}. of Theorem \ref{thm:comparisonmulti} and the conclusion still holds (in particular, the proof we gave still goes through). We will list some possible relaxations below. 

\begin{itemize}
    \item Condition \ref{item:fjLipshitz}): We only use this condition on \eqref{eqn:lipschitz1} and \eqref{eqn:lipschitz2}, and in both situations we are applying it at points $(t, z_1, \ldots, z_a)$ which are very close to the trajectory. In particular, instead of assuming that $f_j$ is $L$-Lipschitz on all of $\mc{D}$,  it suffices to have that $f_j$ is $L$-Lipschitz on the set of points
    \[
   \mc{D}^*:= \Big\{ (t, z_1, \ldots, z_a) \in \mathbb{R}^{a+1}: 0 \le t \le \sigma, y_{j'}(t) \le z_{j'} \le y_{j'}(t) + g(t) \mbox{ for } 1 \le j' \le a \Big\} \subseteq \mc{D}.
    \]

    \item Condition \ref{item:fjbounded}): We only use this condition to conclude that the system \eqref{eq:ysystem} has a unique solution and that those solutions $y_j$ are $B$-Lipschitz. So, instead of Condition \ref{item:fjbounded}) it suffices to just check directly that \eqref{eq:ysystem} has a unique solution that is $B$-Lipschitz.
    \item Condition \ref{item:cooperative}): We only use the cooperativity of the functions to get line \eqref{eq:monotonicity}. In this situation, for our point $(t, z_1, \ldots, z_a)$ some $z_j$ is in its critical interval (and all the $z_i$ are below their upper bounds). In particular, instead of assuming the full Condition \ref{item:cooperative}), for \eqref{eq:monotonicity} it suffices to have that $f_j(t, z_1, \ldots z_a)$ is upper bounded by
\begin{equation*}
 f_j\rbrac{t, \frac{ny_1(t)+g(t)}{n}, \ldots, \frac{ny_{j-1}(t)+g(t)}{n}, z_j, \frac{ny_{j+1}(t)+g(t)}{n}, \ldots, \frac{ny_a(t)+g(t)}{n}}
\end{equation*} 
whenever $z_{j'} \le y_{j'}(t) + g(t)/n$ for all $j'$ and $z_{j} \ge y_j(t)$.

    \item Condition \ref{item:expected_change}.~: We only use this condition to get \eqref{eq:mtg}. In this situation we have that one of the $Z_j$ is in its critical interval (and all the $Z_i$ are below their upper bounds). Instead of  Condition \ref{item:expected_change}.~it suffices to have the following: for each $1 \le j \le a$. If $Z_{j'}(i) \le ny_{j'}(i/n) + g(i/n)$ for all $1 \le j'\le a$ and $Z_j \ge ny_j(i/n)$, then
    $(i/n, Z_1(i)/n, \ldots Z_a(i)/n) \in \scr{D}$ and
    $$\E[ \D Z_j(i) \mid \mc{H}_i] \le  f_j(i/n, Z_1(i)/n, \ldots Z_a(i)/n)  + \delta .$$
\end{itemize}
\end{Remark}


\section{Recovering a Version of Wormald's Theorem}

In this section we recover the standard (two-sided) differential equation method of Wormald \cite{wormald1995}. The statement resembles the recent version given by Warnke \cite{warnke2019wormalds} in the sense that it does not use any asymptotic notation and instead gives explicit bounds for error estimates and failure probabilities. Like Warnke's proof, ours has a {\em probabilistic part} and a {\em deterministic part}. Our probabilistic part is much the same as Warnke's in that we apply a deviation inequality (though we use Freedman's theorem rather than the Azuma-Hoeffding inequality) to the martingale part of a Doob decomposition. That being said, the deterministic part of our argument is quite different than the deterministic part of Warnke's argument. In fact, we were not able to adapt Warnke's argument to the one-sided setting.

Given $a \in \mb{N}$, suppose that $\scr{D} \subseteq \mb{R}^{a+1}$ is a bounded domain,
and for $1 \le j \le a$, let $f_j: \scr{D} \rightarrow \mb{R}$. We assume that the following hold for each $j$:
\begin{enumerate}
    \item $f_j$ is $L$-Lipschitz, and 
    \item $|f_j|\le B$ on $\scr{D}$.
\end{enumerate}
Given $(0,\yt_1, \ldots, \yt_a) \in \scr{D}$, assume that $y_1(t), \ldots, y_a(t)$ is the (unique) solution to the system
\begin{equation*}
    y_j'(t) = f_j(t, y_1(t), \ldots, y_a(t)), \qquad y_j(0)=\yt_j.
\end{equation*}
for all $t \in [0, \sigma]$ where $\sigma$ is a positive value. Note that unlike in \Cref{thm:de_one_sided}, we make a further assumption involving $\sigma$ below in \Cref{thm:de_two_sided}.

\begin{theorem} \label{thm:de_two_sided}
Suppose for each $1 \le j \le a$ we have a sequence of random variables $(Y_j(i))_{i= 0}^\infty$ which
is adapted to the filtration  $(\mc{H}_i)_{i=0}^\infty$. Let $n \in \mb{N}$, and $\beta, b, \lambda, \delta > 0$
be any parameters such that $\lambda \ge \frac{L+BL + \delta n}{3L}$.
Moreover, assume that $\sigma >0$ is any value such that $(t,y_1(t), \ldots, y_a(t))$ has $\ell^{\infty}$-distance at least $3 \lambda e^{2Lt}/n$ from the boundary of $\scr{D}$ for all $t \in [0, \sigma)$.
Given an arbitrary stopping time $I \ge 0$ adapted to $(\mc{H}_i)_{i=0}^\infty$, suppose that the following properties hold for each $ 0 \le i  <  \min\{I, \sigma n \}$:
\begin{enumerate}
    \item The `Boundedness Hypothesis':  $\max_{j} |\Delta Y_j(i)| \le \beta$, and $\max_{j} \mb{E}[ ( \Delta Y_{j}(i) )^2 \mid \scr{H}_i] \le b$. \label{item:boundedness2}
    \item The `Trend Hypothesis': If 
    $(i/n, Y_1(i)/n, \ldots Y_a(i)/n) \in \scr{D}$, then
    $$\Big|\E[ \D Y_j(i) \mid \mc{H}_i] - f_j(i/n, Y_1(i)/n, \ldots Y_a(i)/n)\Big|  \le \delta $$
    
    for each $1 \le j \le a$. \label{item:expected_change2}
    \item  The `Initial Condition':  $|Y_j(0) - y_j(0) n| \le \lambda$ for all $1 \le j \le a$. \label{item:initial2}
\end{enumerate}
Then, with probability at least $1 - 2 a \exp\left( -\frac{\lambda^2}{2( b \sigma n + 2\beta \lambda )} \right)$,
 \begin{equation*}\label{eq:mainprob2}
 |Y_j(i) - ny_{j}(i/n)| \le 3 \lambda e^{2L i/n} 
    \end{equation*}
for all $1 \le j \le a$ and $0 \le i \le \min\{I, \sigma n\}$.

\end{theorem}

We conclude the section with an extension of \Cref{thm:de_two_sided}
analogous to \Cref{cor:assuming_conditions} of \Cref{thm:de_one_sided}. We omit
the proof, as it follows almost identically to the proof of \Cref{cor:assuming_conditions}.
\begin{corollary}[of \Cref{thm:de_two_sided}] \label{cor:assuming_conditions2}
Suppose that in the terminology of \Cref{thm:de_two_sided}, Conditions \ref{item:boundedness2}. and \ref{item:expected_change2}.
are only verified for each $0 \le i \le \min\{I, \sigma n\}$ which satisfies $|Y_{j'}(i) - ny_{j'}(i/n)| \le 3 \lambda e^{2Li/n}$ for all $1 \le j'\le a$. In this case, the conclusion of \Cref{thm:de_two_sided} still holds.
I.e., with probability at least $1 - 2 a \exp\left( -\frac{\lambda^2}{2( b \sigma n + 2\beta \lambda )} \right)$,
 \begin{equation*}
 |Y_j(i) - ny_{j}(i/n)| \le 3 \lambda e^{2L i/n} 
    \end{equation*}
for all $1 \le j \le a$ and $0 \le i \le \min\{I, \sigma n\}$.
\end{corollary}

\begin{proof}[Proof of \Cref{thm:de_two_sided}] 
Fix $0 \le i \le \sigma n$, and again set $m:= \sigma n$, $t=t_i=i/n$, and $g(t):= 3\lambda e^{2Lt}$
for convenience. Define
\begin{align*}
   S^{\pm}_j(i)&:=Y_j(i)-(ny_j(t) \pm g(t)),\\
   X^{\pm}_j(i)&:=\sum_{k=0}^{i-1} \E[\D S^{\pm}_j(k) \mid \mc{H}_{k}],\\
   M^{\pm}_j(i)&:= S^{\pm}_j(0) + \sum_{k=0}^{i-1} \rbrac{\D S^{\pm}_j(k)- \E[\D S^{\pm}_j(k) \mid \mc{H}_{k}]}
\end{align*}
so that  $(X^{\pm}_j(i) + M^{\pm}_j(i))_{i= 0}^{m}$ is the Doob decomposition of $(S^{\pm}_j(i))_{i= 0}^m$. Note that
\[
\D S^{\pm}_j(k)- \E[\D S^{\pm}_j(k) \mid \mc{H}_{k}] = \D Y_j(k)- \E[\D Y_j(k) \mid \mc{H}_{k}],
\]
and so $M^+_j(i)$ is almost the same as $M^-_j(i)$. More precisely, we have
 \[
 M^{\pm}_j(i) = M_j(i) \mp g(0)
 \]
 where
 \[
M_j(i):= Y_j(0)-ny_j(0) + \sum_{k=0}^{i-1} \rbrac{\D Y_j(k)- \E[\D Y_j(k) \mid \mc{H}_{k}]}
 \]
(which is also a martingale). As in the proof of Theorem \ref{thm:de_one_sided}, we have 
 $|\D M_j(i)| \le 2\beta$ and  $\Var[ \D M_{j}(i) \mid \scr{H}_i] \le b$. Now, by Theorem \ref{thm:freedman} we have that
 \begin{equation}\label{eq:azuma2}
\Pr\Big(\exists\; 0 \le j \le a, 0 \le i \le m \mbox{ such that } |M_j(i) - M_j(0)| \ge \lambda\Big) \le 2 a \exp\left( -\frac{\lambda^2}{2( b m + 2\beta \lambda)} \right).
\end{equation}
Suppose that the event above does not happen, so for all $0 \le j \le a,0 \le i \le m$ we have that $|M_j(i) - M_j(0)| < \lambda$. We will show that we also have $|Y_j(i) - ny_j(t)|\le g(t)$ for all $0 \le i \le m$. Note that $y_j$ is $B$-Lipschitz as before. Define the \tbf{critical interval}
\[
I_j(i) := \sbrac{ny_j(t)-g(t), ny_j(t)+g(t)}.
\]
Suppose for the sake of contradiction that $i'$ is minimal with $0 \le i' \le m$ and $Y_j(i') \notin I_j(i')$ for some $j$. We will consider the case where $Y_j(i') > ny_j(t)+g(t)$ (the case where $Y_j(i') < ny_j(t)-g(t)$ is handled similarly with some inequalities reversed). In other words, $S^+_j(i') > 0$. 
First observe that since $g(0):=3 \lambda$, Condition \ref{item:initial2}. implies $S^+_j(0) \le -2\lambda$. In particular, $i' >0$ and 
\begin{equation}\label{eq:Sdiff}
    S^+_j(i') - S^+_j(0)  >2\lambda.
\end{equation}

For $0\le i < i'$, we have (explanation follows)
\begin{align*}
    \E[\Delta S^+_{j}(i) \mid \mc{H}_i] & = \E[ \D Y_{j}(i) \mid \mc{H}_i] - \D ny_j(t) - \D g(t)\nn\\
    & \le f_j(t, Y_1(i)/n, \ldots Y_a(i)/n) + \delta - f_j(t, y_1(t), \ldots, y_a(t)) + (L+BL)n^{-1} -  n^{-1} g'(t)\nn\\
    & \le Ln^{-1}g(t) + \delta  +(L+BL)n^{-1} -  n^{-1} g'(t)\nn\\
    & \le  -[3L \lambda  - (L+BL + \delta n)]n^{-1}\\
    & \le 0.
\end{align*}
Indeed, the first line is by definition and the second line follows just like \eqref{eq:mtg},
with the minor caveat that we now must show that $(t, Y_1(i)/n, \ldots ,Y_a(i)/n)$
is within the domain $\scr{D}$ to apply Condition \ref{item:expected_change2} (recall that in \Cref{thm:de_one_sided} this is simply assumed). Observe that by the definition of $\sigma$, $(t, y_1(t), \ldots , y_{a}(t))$ is in $\scr{D}$ \textit{and} at least $\ell^{\infty}$-distance $g(t)/n$ from the boundary of $\scr{D}$. On the other hand, since $Y_{j'}(i) \in I_{j'}(i)$ for all $1 \le j' \le a$, we know
that $|Y_{j'}(i)/n - y_{j'}(t)| \le g(t)/n$ for all $1 \le j' \le a$. Thus, $(t, Y_1(i)/n, \ldots Y_a(i)/n) \in \scr{D}$, and so
$$\E[ \D Y_{j}(i) \mid \mc{H}_i] \le  f_j(t, Y_1(i)/n, \ldots ,Y_a(i)/n) + \delta$$ by Condition \ref{item:expected_change2}.
The remaining justification is much the same as before. Since $f_j$ is $L$-Lipschitz and $|Y_{j'}(i) - ny_{j'}(t)|\le g(t)$ for all $j'$, we have $f_j(t, Y_1(i)/n, \ldots Y_a(i)/n) \le f_j(t, y_1(t), \ldots, y_a(t))+Ln^{-1}g(t)$ and the third line follows. The fourth and fifth lines follow just as \eqref{eq:14} and \eqref{eq:15}. Therefore, for $0 \le i < i'$ we have that
\[
0 \ge \E[\Delta S^+_j(i) \mid \mc{H}_i] = \E[\Delta X^+_j(i) \mid \mc{H}_i] + \E[\Delta M^+_j(i) \mid \mc{H}_i] = \Delta X^+_j(i)
\]
since $(M^+_j(i))_{i= 0}^\infty$ is a martingale and $\D X^+_j(i)$ is $\mc{H}_{i}$-measurable. In particular 
\begin{equation}\label{eq:Xdiff}
    X^+_j(i') \le X^+_j(0).
\end{equation} 
But now we can derive our final contradiction (explanation follows):
\begin{align*}
  2 \lambda &< S^+_j(i') - S^+_j(0) \\
  & = X^+_j(i') - X^+_j(0) + M_j(i') - M_j(0)\\
  &\le \lambda. 
\end{align*}
Indeed, the first line is from \eqref{eq:Sdiff}, the second line is by the Doob decomposition, and the last follows from \eqref{eq:Xdiff} and our assumption that the event described on line \eqref{eq:azuma2} does not happen. Of course $2\lambda < \lambda$ is a contradiction and we are done. 
\end{proof}

\bibliographystyle{plain}

\bibliography{ref.bib}

\appendix

\end{document}